\documentclass{amsart}
\usepackage[utf8]{inputenc}
\usepackage{amsmath, amssymb, amsfonts, amsthm}
\usepackage{enumerate}
\usepackage{bbold}
\usepackage{color}
\usepackage{graphicx}
\usepackage{fullpage}
\usepackage{tikz}
\usepackage{blkarray, bigstrut} %
\usepackage{tikz-cd}
\usepackage{hyperref}
\usepackage{lipsum}

\usepackage{bbm}

\allowdisplaybreaks
\DeclareMathOperator{\Rep}{\textrm{Rep}}

\DeclareMathOperator{\Hom}{\textrm{Hom}}

\newcommand{\q}[1]{[#1]_q}

\newcommand{\att}[2]{\raisebox{-.5\height}{ \includegraphics[scale = #2]{diagrams/#1.pdf}}}
\usepackage{array}   
\newcolumntype{L}{>{$}l<{$}}

\theoremstyle{plain}
\newtheorem{thm}{Theorem}[section]

\newtheorem{cor}[thm]{Corollary}

\theoremstyle{definition}
\newtheorem{defn}[thm]{Definition}

\newtheorem{remark}[thm]{Remark}

\title{The construction of a $E_7$-like quantum subgroup of $SU(3)$}
\author{Cain Edie-Michell and Lance Marinelli}
\address{Cain Edie-Michell\\
University of New Hampshire\\
Durham, 
New Hampshire}
\email{cain.edie-michell@unh.edu}
\address{Lance Marinelli\\
University of New Hampshire\\
Durham, 
New Hampshire}
\email{lance.marinelli@unh.edu}
\date{}

\begin{document}

\maketitle
\begin{abstract}
In this short note we construct an embedding of the planar algebra for $\overline{\Rep(U_q(\mathfrak{sl}_3))}$ at $q = e^{2\pi i \frac{1}{24}}$ into the graph planar algebra of di Francesco and Zuber's candidate graph $\mathcal{E}_4^{12}$. Via the graph planar algebra embedding theorem we thus construct a rank 11 module category over $\overline{\Rep(U_q(\mathfrak{sl}_3))}$ whose graph for action by the vector representation is $\mathcal{E}_4^{12}$. This fills a small gap in the literature on the construction of $\overline{\Rep(U_q(\mathfrak{sl}_3))}$ module categories. As a consequence of our construction, we obtain the principal graphs of subfactors constructed abstractly by Evans and Pugh. 
\end{abstract}

\section{Introduction}
To every module category over a modular tensor category (MTC), there is an associated \textit{modular invariant}. This is a positive integer valued matrix commuting with the $SL(2, \mathbb{Z})$ representation of the MTC. These modular invariants are a useful tool for studying module categories, and have played a key role in classification efforts. However, the modular invariant is not a complete invariant. There are many examples of modular invariants which do not come from module categories \cite{UnPhys}, and also of distinct module categories with the same modular invariant \cite[Sections 11 and 12]{SU3}. A modular invariant is referred to as  \textit{physical} if it is realised by a module category. Even in the situation where a modular invariant is known to be physical, it can be difficult to determine the structure of the corresponding module categories.

A large class of MTC's come from the (semisimplified) representation theory of quantum groups at roots of unity \cite[Chapter 7]{Bak}. These categories are typically denoted $\overline{\Rep(U_q(\mathfrak{g}))}$. In the special case of the Lie algebra $\mathfrak{sl}_3$, the modular invariants were classified by Gannon \cite{SU3Mod}. In work of Evans and Pugh \cite{SU3}, all of the $SU(3)$ modular invariants were shown to be physical. For all bar one modular invariant, their proof was via explicit construction of the corresponding module categories (using Ocneanu cell systems). The remaining modular invariant was shown to be physical via a relative tensor product construction. As the relative tensor product of module categories is a difficult construction to work with in practice, the explicit structure of the corresponding module category has not been confirmed. It should also be noted that in \cite[Section 5.4]{SU3Sub} some structure of this module category is deduced based on an assumption on its corresponding algebra object. Further, in \cite{Ocneanu}, an explicit construction of this module category is claimed without detail.


The modular invariant in question can be found in \cite{SU3Mod} labelled as $\left(\mathcal{E}^{(2)}_{9}\right)^c$. There has been some work on deducing the module fusion graph (the graph representing the action of $\Lambda_1$ 
on the module) for the module category category corresponding to this modular invariant. In \cite{DiFran} Di Francesco and Zuber suggest the following graph (with some physical supporting evidence): 
\[\mathcal{E}_4^{12}\quad := \quad \att{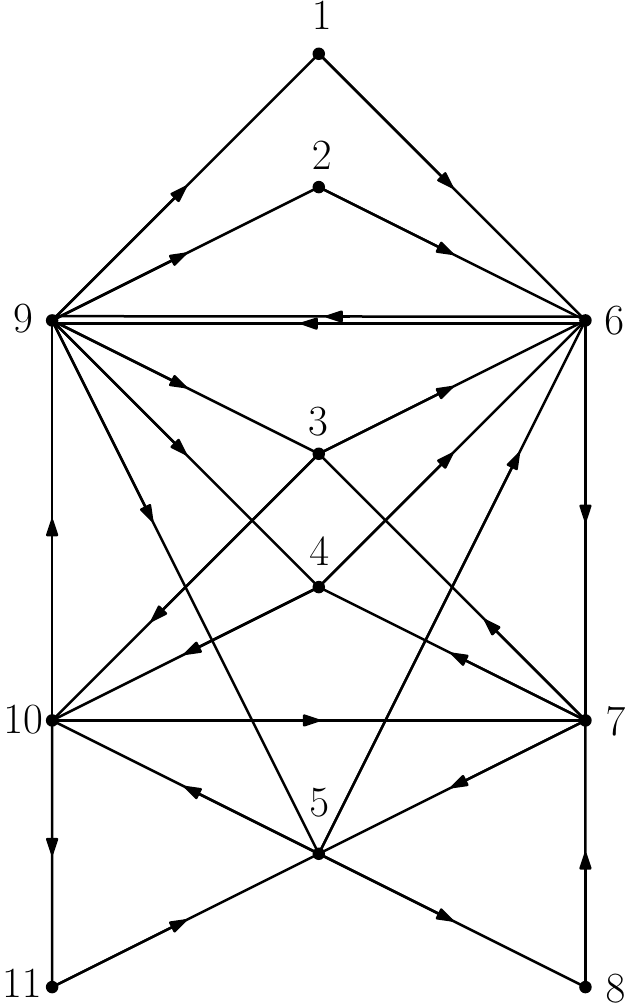}{.29}.\]
As it will be useful throughout this paper, the Frobenius-Perron eigenvector for this graph is
\[\lambda = \left\{\frac{\q{5}}{\q{3}},\frac{\q{5}}{\q{3}},  \frac{\q{2}\q{4}}{\q{3}} ,  \frac{\q{2}\q{4}}{\q{3}} ,\q{3}  ,\q{5},\q{3},1,\q{5},\q{3},1   \right\}.\]
In this paper, we fix a small gap in the literature by explicitly constructing a module category with module fusion graph $\mathcal{E}_4^{12}$. Our technique for constructing this module category is by using the graph planar algebra embedding theorem \cite[Theorem 1.3]{EH}. The use of this technique is typically been referred to as \textit{cell systems} in the context of quantum groups \cite{Ocneanu, SU3, EvansSO3}. More precisely, we find the following element of $oGPA(\mathcal{E}_4^{12})$. We direct the reader to Subsection~\ref{sub:GPA} for the definition of $oGPA(\mathcal{E}_4^{12})$.
\begin{defn}\label{def:main}
  Let $q = \zeta_{24}$, and $z$ the root of the polynomial $9x^{16} - 14 x^8 +9$ with numerical value closest to $-0.996393+0.0848571 i$. We define $W\in \Hom_{oGPA(\mathcal{E}_4^{12})}(-\to ++)$ as the functional defined on basis elements by 
    \begin{align*}
  &W_{1,6,9} =\begin{cases} \sqrt{\q{2}} \quad &6\xrightarrow{\alpha} 9\\
                             0 \quad &6\xrightarrow{\beta}9\\
  \end{cases}
  && W_{2,6,9}=\begin{cases} z^{-1}\sqrt{ \frac{1}{\q{2}}} \quad &6\xrightarrow{\alpha} 9\\
                        \zeta_{24}^{19} \sqrt{ \frac{\q{3}}{\q{2}}}\quad &6\xrightarrow{\beta} 9
  \end{cases}
  && W_{3,6,9}=\begin{cases} z \sqrt{ \frac{1}{\q{2}}}\quad &6\xrightarrow{\alpha} 9\\
 \zeta_3 z\sqrt{\frac{\q{3}}{\q{4}(\q{2}+\q{3})}  } \quad &6\xrightarrow{\beta} 9\\
  \end{cases}\\
 & W_{4,6,9}=\begin{cases} \sqrt{ \frac{1}{\q{2}}} &6\xrightarrow{\alpha} 9 \\
 \zeta_8^5\sqrt{ \frac{\q{3}(\q{2} + \q{3})}{\q{4}\q{5}} } &6\xrightarrow{\beta} 9 
  \end{cases}
  && W_{5,6,9} =\begin{cases}
       \mathbf{i}z^{-1}\sqrt{ \frac{1}{\q{2}}} &6\xrightarrow{\alpha} 9 \\
        \zeta_{48}^{11}z\sqrt{ \frac{\q{4}}{\q{5}}}&6\xrightarrow{\beta} 9 
   \end{cases} \\
  & W_{3,6,7} = z\sqrt{ \frac{\q{2}(1 +\q{2})}{\q{4}}} 
   && W_{3,10,7} = z\sqrt{\frac{\q{2}^2}{\q{4}(\q{2}+\q{3})}}  
     &&  W_{3,10,9} = z\sqrt{ \frac{\q{3}(1 +\q{2})}{\q{2}\q{4}}} \\
     &  W_{4,6,7} = \zeta_8^5\sqrt{ \frac{\q{2}\q{3}}{\q{4}(1+ \q{2})}  } 
 && W_{4,10,7} =z \sqrt{\frac{\q{2}+\q{3}}{\q{4}} } 
   && W_{4,10,9}=z\sqrt{\frac{\q{3}^2}{\q{2}\q{4}(1+\q{2})} } \\
 & W_{5,6,7} = \zeta_8 \sqrt{\frac{\q{2}}{\q{3}} } 
  &&W_{5,8,7} = z\sqrt{\frac{\q{2}}{\q{3}} } 
 && W_{5,10,7} =z \sqrt{\frac{1}{\q{2}} } \\
  &W_{5,10,9} =z \sqrt{\frac{1}{\q{2}} } 
   &&W_{5,10,11} =z \sqrt{\q{2} } 
\end{align*}
with the remaining values on basis elements defined by the rotational formula $W_{a,b,c} = \sqrt{\frac{\lambda_{b}}{\lambda_c}}W_{b,c,a}$. Here we use the notation that $\zeta_\ell := e^{2\pi i \frac{1}{\ell}}$.
\end{defn}
Our main result shows that this distinguished element satisfies the relations required to give an embedding for the planar algebra of $\overline{\Rep(U_q(\mathfrak{sl}_3))}$ associated to the object $\Lambda_1$.
\begin{thm}\label{thm:main}
The map 
\[\att{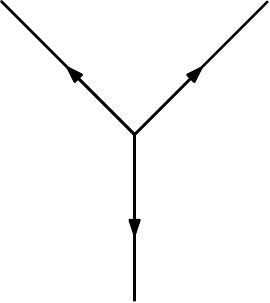}{.15} \mapsto W \in\Hom_{oGPA(\mathcal{E}_4^{12})}(-\to ++)\]
defines a tensor functor 
\[\mathcal{P}_{\overline{\Rep(U_q(\mathfrak{sl}_3))}; \Lambda_1}\to oGPA(\mathcal{E}_4^{12}).\]
\end{thm}
The graph planar algebra embedding theorem \cite[Theorem 1.3]{EH} (and \cite[Theorem 1.1]{dan} for the slight technical alteration needed for our set-up) then gives the construction of the module category.
\begin{cor}\label{cor:main}
    There exists a module category $\mathcal{M}$ over $\overline{\Rep(U_q(\mathfrak{sl}_3))}$ such that the action graph for $\Lambda_1$ is $\mathcal{E}_4^{12}$.
\end{cor}

As shown in \cite{SU3Sub}, we obtain several subfactors of the hyperfinite $\textrm{II}_1$ factor $\mathcal{R}$ as a consequence of Corollary~\ref{cor:main}. The subfactor with smallest index (= $24 \left(2+\sqrt{3}\right)$) has principal graph 
\[\att{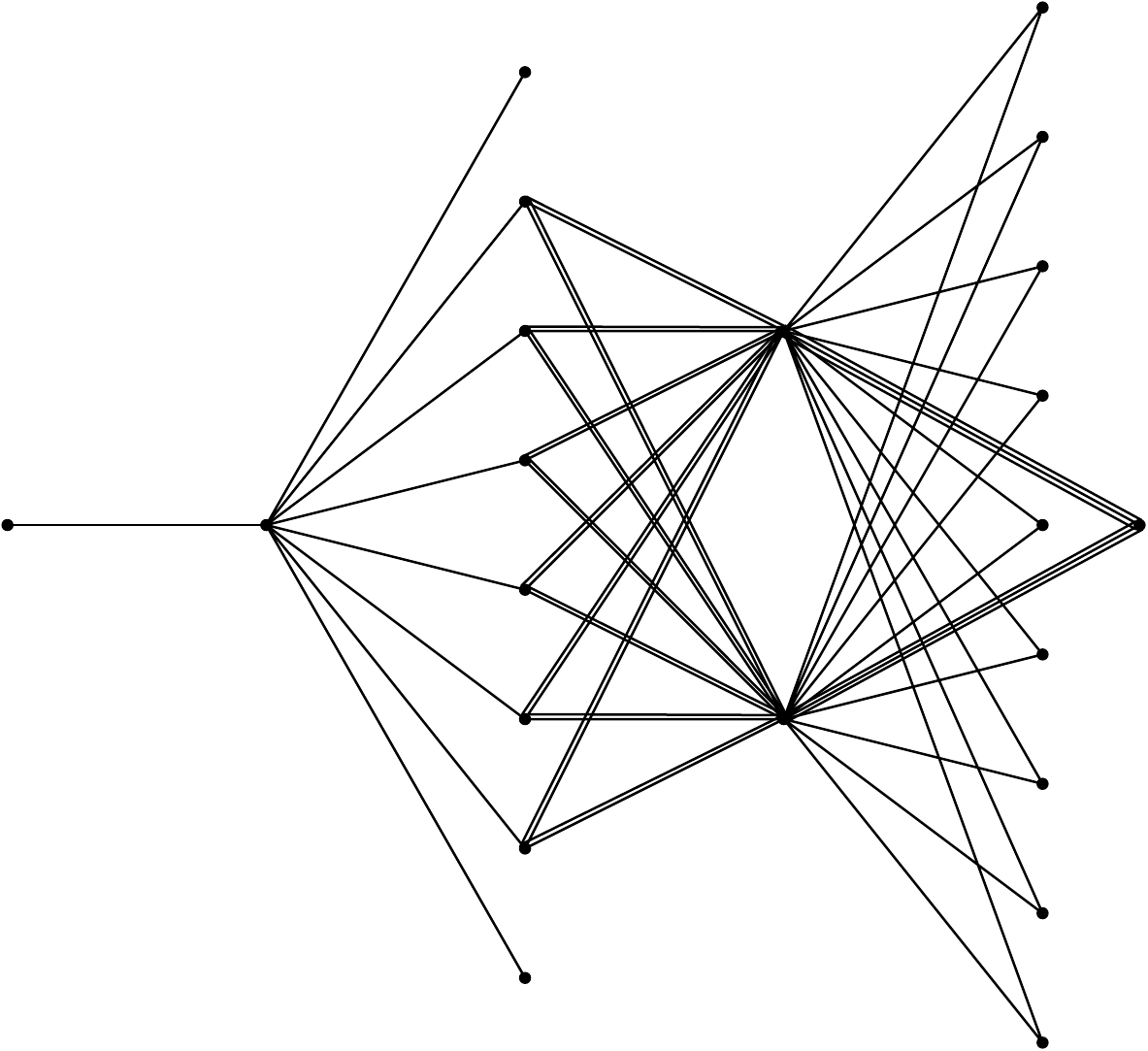}{.4}\]
The above principal graph is obtained from the graph $\mathcal{E}_4^{12}$ via the equations of \cite[Section 7]{dan}.

Our strategy for obtaining the embedding in Definition~\ref{def:main} is low-brow, but effective. We begin by numerically approximating a solution for the embedding of the element $\q{2}\cdot p_{\Lambda_2}\in \overline{\Rep(U_q(\mathfrak{sl}_3))}$ into $oGPA(\mathcal{E}_4^{12})$. As the element $\q{2}\cdot p_{\Lambda_2}$ satisfies the Hecke algebra relations, the equations governing its embedding into $oGPA(\mathcal{E}_4^{12})$ are polynomial (of max degree 3), and are amenable to numerical approximation. From this numerical approximation we can then guess exact values for most of the coefficients of the embedding. With many of the coefficients exactly determined, many of the polynomial equations governing the embedding are now linear, and can be solved exactly. This gives us a candidate for the embedding of the element $\q{2}\cdot p_{\Lambda_2}$. Using the techniques developed in \cite{dan}, we can then determine the embedding of $\att{tri1}{.15}\in\overline{\Rep(U_q(\mathfrak{sl}_3))}$ into $oGPA(\mathcal{E}_4^{12})$ presented in Definition~\ref{def:main}.

While our method of numerical approximation is ad-hoc, we have several arguments to justify our approach. The first is that the results of \cite{ModPt1} show that there is at most one embedding of $\overline{\Rep(U_q(\mathfrak{sl}_3))}$ into $oGPA(\mathcal{E}_4^{12})$. Thus we don't have to worry about missing potential solutions in our numerical approximation. The second argument is that the modular invariant we are studying lives in the finite family of ``$E_7$-like modular invariants'' again by the results of \cite{ModPt1}. In fact, there are only six physical modular invariants in this family up to level-rank duality. As such, we consider ad-hoc techniques fair game for explicitly constructing the module categories corresponding to these modular invariants.

\subsection*{Acknowledgements}
The first author was partially supported by NSF grant DMS 2245935. The second author was fully supported by NSF grant DMS 2245935 as an undergraduate researcher.

\section{Preliminaries}

We refer the reader to \cite{Book} for the basics on tensor categories and module categories. We refer the reader to \cite{Sawin,primer} for the basics on the tensor categories $\overline{\Rep(U_q(\mathfrak{g}))}$.

\subsection{Planar Algebras}

In this subsection we define planar algebras, and explain how they are can be obtained from a pivotal tensor category. We refer the reader to \cite{PA1} for additional details and examples.

For this paper we will use the non-standard categorical definition of a planar algebra.
\begin{defn}
    A planar algebra $\mathcal{P}$ is a strictly pivotal monoidal category whose objects are strings in the alphabet $\{+, -\}$.
\end{defn}
\begin{remark}
    Note that our definition of a planar algebra does not require direct sums of objects, nor that idempotents split.
\end{remark}

The standard definition of a planar algebra by action of the planar operad \cite[Section 1]{PA1} is equivalent to our non-standard definition. The action of a compatible planar tangle $T$ on a collection of morphisms $f_i \in \Hom_\mathcal{P}(s_i \to t_i)$ is given by inserting the morphisms $f_i$ into the planar tangle $T$, and evaluating the diagram as a morphism of $\mathcal{P}$ using the pivotal graphical calculus \cite[Chapter 2]{Tur}. We direct the reader to \cite[Chapter 3]{scottGPA} for detailed notes on this correspondence.

A large class of planar algebras can be constructed as special subcategories of pivotal tensor categories. These special subcategories are defined as follows.

\begin{defn}
  Let $\mathcal{C}$ be a pivotal tensor category, and $X\in \mathcal{C}$ an object. We define the planar algebra generated by $X$, which we denote $\mathcal{P}_{\mathcal{C}, X}$ as follows. The objects of $\mathcal{P}_{\mathcal{C}, X}$ are sequences in $\{+, -\}$. Let $s,t$ be two objects. We define
  \[       \operatorname{Hom}_{\mathcal{P}_{\mathcal{C}, X}}(s\to t) :=   \operatorname{Hom}_{\mathcal{C}}(X^{s_1}\otimes X^{s_2} \otimes \cdots \to X^{t_1}\otimes X^{t_2}  \otimes \cdots  )  \]
  where we understand $X^+ = X$ and $X^- = X^*$.
\end{defn}

If the object $X$ Cauchy tensor generates $\mathcal{C}$ (in the sense of \cite{EH}), then $\mathcal{P}_{\mathcal{C}, X}$ contains a projection onto every simple object of $\mathcal{C}$. Hence the Cauchy completion of $\mathcal{P}_{\mathcal{C}, X}$ is monoidally equivalent to $\mathcal{C}$. In this sense, the subcategory $\mathcal{P}_{\mathcal{C}, X}$ remembers all the information of the original category $\mathcal{C}$, while being significantly simpler.

An important example of a planar algebra is the Kazhdan-Wenzl presentation for $\operatorname{Rep}(U_q(\mathfrak{sl}_N))$. Let $X = \Lambda_1$ be the vector representation. The planar algebra $\mathcal{P}_{\operatorname{Rep}(U_q(\mathfrak{sl}_N)), \Lambda_1}$ is then described in \cite{SovietHans} via generators and relations. The generators of this planar algebra are
\[  \raisebox{-.5\height}{ \includegraphics[scale = .3]{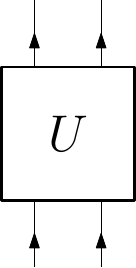}}:=[2]_q\raisebox{-.5\height}{ \includegraphics[scale = .3]{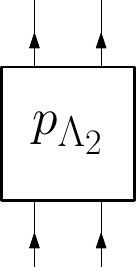}} \quad  \in \quad  \mathcal{P}_{\operatorname{Rep}(U_q(\mathfrak{sl}_N)),\Lambda_1}( +^2 \to +^2 )  \quad \text{and} \quad \att{triv}{.25} \quad \in \quad  \mathcal{P}_{\operatorname{Rep}(U_q(\mathfrak{sl}_N)),\Lambda_1}( \mathbf{1} \to +^N  )\]
The planar algebra is then constructed as the free planar algebra built from the generating morphisms (allowing duality morphisms, along with tensor products, compositions, and sums of these morphisms), modulo the generating relations. We have the following relations between the generators (which are sufficient when $q = e^{2\pi i \frac{1}{N+k}}$ for some $k\in \mathbb{N}$ by \cite{dan})
\begin{align*}
(\textrm{R1}) :  &\quad\raisebox{-.5\height}{ \includegraphics[scale = .3]{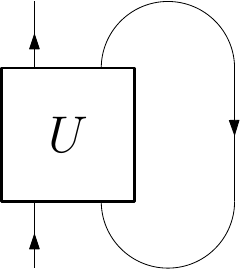}} = \raisebox{-.5\height}{ \includegraphics[scale = .3]{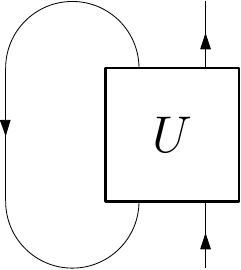}}= [N-1]_q\raisebox{-.5\height}{ \includegraphics[scale = .3]{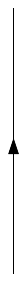}}  & (\textrm{R2}) :  &\quad \left(\raisebox{-.5\height}{ \includegraphics[scale = .3]{diagrams/UU.pdf}}\right)^\dag = \raisebox{-.5\height}{ \includegraphics[scale = .3]{diagrams/UU.pdf}}\\
(\textrm{R3}) :  &\quad\raisebox{-.5\height}{ \includegraphics[scale = .3]{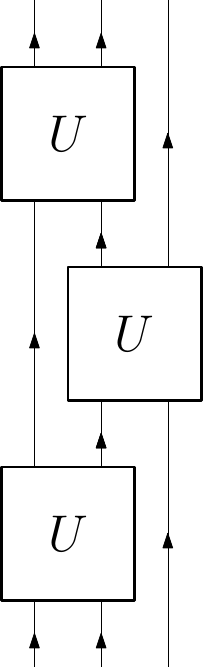}}  -\raisebox{-.5\height}{ \includegraphics[scale = .3]{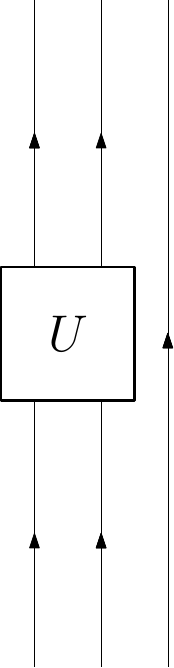}}  =\raisebox{-.5\height}{ \includegraphics[scale = .3]{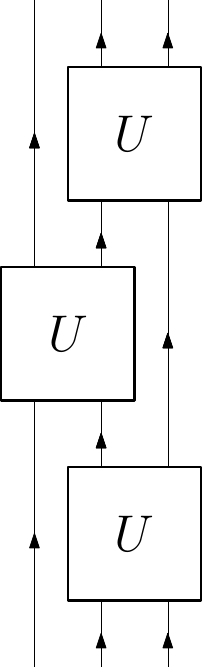}}  -\raisebox{-.5\height}{ \includegraphics[scale = .3]{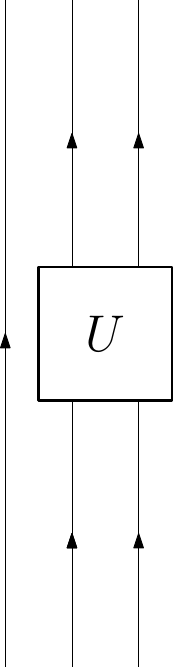}} &(\textrm{Hecke}) :  &\quad\raisebox{-.5\height}{ \includegraphics[scale = .3]{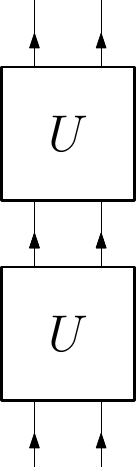}}= [2]_q\raisebox{-.5\height}{ \includegraphics[scale = .3]{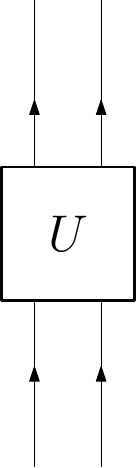}}\\
(\textrm{BA}) :  &\quad \raisebox{-.5\height}{ \includegraphics[scale = .3]{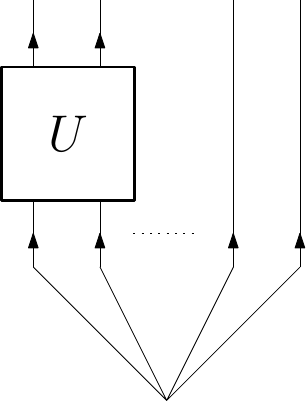}} = [2]_q \raisebox{-.5\height}{ \includegraphics[scale = .3]{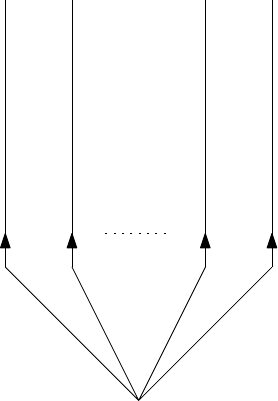}} &(\textrm{RI}) :  &\quad\raisebox{-.5\height}{ \includegraphics[scale = .3]{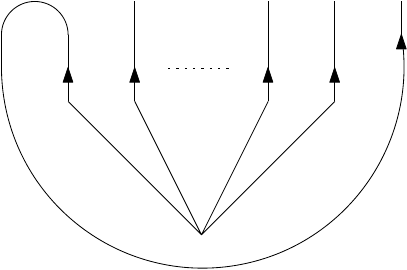}} =  (-1)^{N+1}\raisebox{-.5\height}{ \includegraphics[scale = .3]{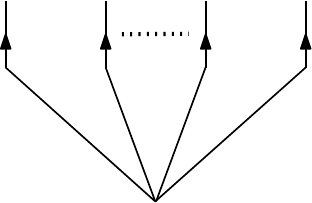}}\\
(\textrm{U}) :  &\quad \raisebox{-.5\height}{ \includegraphics[scale = .3]{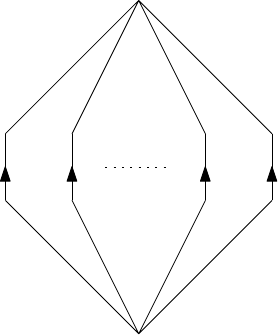}}=1
\end{align*}

Note that for this paper, we will be specialised to the case where $N=3$ and $q = e^{2\pi i \frac{1}{24}}$.
\subsection{The graph planar algebra}\label{sub:GPA}
A key planar algebra used in this paper is the graph planar algebra constructed from a graph $\Gamma$ and a Frobenius-Perron eigenvector $\lambda$ for $\Gamma$. The construction of the graph planar algebra is due to Jones \cite{OGGPA}. 

\begin{defn}
    Let $\Gamma$ be a graph, and $\lambda$ a positive Frobenius-Perron eigenvector for $\Gamma$. We define the planar algebra $oGPA(\Gamma)$ as follows. For two strings $s$ and $t$ in $\{+,-\}$ we define
    \[ \operatorname{Hom}_{oGPA(\Gamma)}(s \to t)  \cong \operatorname{span}_{\mathbb{C}}\{    (p,q) : p \text{ is a $s$ path}, q \text{ is a $t$ path}, s(p) = s(q), t(p) = t(q)\}.\]  
    The composition and tensor product operators are defined on basis elements by
    \begin{align*}
    (p',q')\circ (p,q) &= \delta_{q',p} (p',q)\\
    (p,q) \otimes (p',q') &= \delta_{t(p), s(p')}\delta_{t(q), s(q')}(pp', qq')
\end{align*}
which are extended linearly. We define distinguished rigidity maps by
\begin{align*}
\operatorname{ev}_{(+,-)}   :=   \sum_{\text{$(e,\overline{e})$ a $(+,-)$-path}} \sqrt{\frac{\lambda_{t(e)}}{\lambda_{s(e)}}}((e,\overline{e}),s(e)):   (+,-) \to \mathbbm{1}\\
\operatorname{coev}_{(-,+)}   :=   \sum_{\text{$(\overline{e},e)$ a $(-,+)$-path}} \sqrt{  \frac{\lambda_{s(e)}}{\lambda_{t(e)}}}( t(e),(\overline{e},e)):  \mathbbm{1} \to  (-,+)  \\
\operatorname{ev}_{(-,+)}   :=   \sum_{\text{$(\overline{e},e)$ a $(-,+)$-path}}  \sqrt{\frac{\lambda_{s(e)}}{\lambda_{t(e)}}}((\overline{e},e), t(e)):   (-,+) \to \mathbbm{1}\\
\operatorname{coev}_{(+,-)}   :=   \sum_{\text{$(e,\overline{e})$ a $(+,-)$-path}} \sqrt{\frac{\lambda_{t(e)}}{\lambda_{s(e)}}}( s(e),(e,\overline{e})):  \mathbbm{1} \to  (+,-).
\end{align*}
\end{defn}
It is well-known (and straightforward to verify) that $oGPA(\Gamma)$ satisfies the conditions to be a pivotal monoidal category, and hence is a planar algebra. 

The planar algebra $oGPA(\Gamma)$ can be equipped with a $\dag$ structure given by the anti-linear extension of
\[   (p,q)^\dag = (q,p).  \]
With this dagger structure, $oGPA(\Gamma)$ is unitary. We refer the reader to \cite[Section 2.2]{Haag} and \cite[Section 2.2]{dan} for more details on the category $oGPA(\Gamma)$.

The graph planar algebra is useful for this paper due to the graph planar algebra embedding theorem \cite[Theorem 1.3]{EH} (see \cite[Theorem 1.1]{dan} for the slight alteration needed for the non-self-dual setting). This result shows that module categories over a pivotal tensor category are classified by embeddings of an associated planar algebra into graph planar algebras. This allows us to obtain Corollary~\ref{cor:main} from Theorem~\ref{thm:main}.
\section{Finding the solution}

Our first goal is to find an embedding of the element $\att{UU}{.25} \in \Rep(U_q(\mathfrak{sl}_3))$ into $oGPA(\mathcal{E}^{12}_4)$. That is, we need find an element of $\Hom_{oGPA(\mathcal{E}^{12}_4)}(++\to ++)$ satisfying relations (R1), (R2), (R3), and (Hecke). From the definition of the graph planar algebra, this is exactly the assignment of a complex variable to each pair $(v_1 \xrightarrow{\gamma_1} v_2 \xrightarrow{\gamma_2} v_3 ,v_1 \xrightarrow{\gamma_3} v_4 \xrightarrow{\gamma_4} v_3)$ of paths in $\mathcal{E}^{12}_4$, satisfying various polynomial equations which ensure the four above relations hold.

We will use the ``Boltzmann Weight'' notation (as in \cite{SU3}) to represent the coefficient attached to a pair $(v_1 \xrightarrow{\gamma_1} v_2 \xrightarrow{\gamma_2} v_3 ,v_1 \xrightarrow{\gamma_3} v_4 \xrightarrow{\gamma_4} v_3)$. This coefficient will be represented as $U^{v_1, {}^{\gamma_1}v_2^{\gamma_2}}_{{}^{\gamma_3}v_4^{\gamma_4},v_3}$. In the case that there is a single edge between two vertices, the label denoting the edge will be unambiguously suppressed. 

This notation gives a compact matrix notation for presenting our element of $\Hom_{oGPA(\mathcal{E}^{12}_4)}(++\to ++)$. For fixed vertices $v_1, v_3$ we have the matrix
\[U^{v_1}_{\quad v_2}:= \begin{bmatrix}
U^{v_1, {}^{\gamma_a}v_i^{\gamma_b}}_{{}^{\gamma_c}v_j^{\gamma_d},v_2}
\end{bmatrix}_{        {}^{\gamma_a}v_i^{\gamma_b}   ,    {}^{\gamma_c}v_j^{\gamma_d}     }\]
with row and columns indexed by all paths of length 2 beginning at $v_1$ and ending at $v_2$. 

To find the embedding of $\att{UU}{.25}$ into $oGPA(\mathcal{E}^{12}_4)$, we have to solve a large system of polynomial equations. This proved too hard for us to do directly. Instead, we begin by finding a numerical solution, with the end goal being to guess the exact solution. An issue immediately arises in that the gauge group of the embedding is $U(2)\oplus U(1)^{23}$. A numerical solution is going to find a random point in this solution space. This means that any numerical solution we find will be unrecognisable as an exact solution. To fix this we note that the matrix $U^1_{\quad 9}$ is a $2 \times 2$ projection satisfying $(U^1_{\quad 9})^2 = \q{2} \cdot U^1_{\quad 9}$ by (Hecke), and with trace $\q{2}$ by \cite[Lemma 5.6]{dan}. This means we can unitarily conjugate $U^1_{\quad 9}$ by an element of $U(2)$ to arrange that 
\[U^1_{\quad 9}= \begin{blockarray}{cc}
6^\alpha & 6^\beta \\
\begin{block}{[cc]}
\q{2} & 0 \\
0 & 0\\
\end{block}
\end{blockarray}\]
This uses up the $U(2)$ degree of freedom, up to the $U(1)\oplus U(1)$ diagonal subgroup of this $U(2)$. Thus with this fixed choice of $U^1_{\quad 9}$ as above, we have a gauge group of $U(1)^{25}$. In particular, this means that the absolute values of the coefficients in our solution are now fixed.

We now numerically approximate a solution for the remaining coefficients. As expected, the phases on these coefficients are unrecognisable (as the numerical approximation picks out a random point in the solution space $U(1)^{25}$). However, many of the absolute values (which are invariant under the action of $U(1)^{25}$) of our numerical coefficients can be immediately identified. The distinct numerical values in our numerical solution for which we can make guesses for their exact values are as follows:

\begin{center}\resizebox{\hsize}{!}{\begin{tabular}{c | c c c c c}\hline\hline
    Numerical Value & 0& 0.175067 & 0.207107 &0.239146 & 0.341081 \\ \hline
  Exact Guess      & 0&  $\frac{1}{\q{4}}\left(\q{2} + \frac{\q{3}^2}{\q{5}} \right)-1$& $\frac{\q{2}}{\q{3}}\left(1 + \frac{\q{2}}{\q{3}} \right)-1$&$\frac{\q{3}}{\q{4}}\left(1 + \frac{1}{\q{2}} \right)-1$&$\frac{1}{\q{3}}\left(\q{2} + \frac{\q{4}}{\q{3}} \right)-1$\\ \hline\hline
   Numerical Value & 0.366025& 0.393847 & 0.439158 &0.481717 & 0.5 \\\hline
  Exact Guess      & $\frac{1}{\q{3}}$& $\frac{1}{\q{4}}\left(\q{2} + \q{3} \right)-1$& $\frac{\q{2}}{\q{3}}+\frac{\q{3}}{\q{5}}-1$&$\sqrt{\frac{\q{4}}{\q{2}\q{3}^2}}$&$\frac{1}{2}$\\ \hline\hline
  Numerical Value & 0.517638& 0.538005 & 0.605 &0.68125  & 0.707107 \\\hline
  Exact Guess      & $\frac{1}{\q{2}}$& $\frac{1}{\q{4}}\left(\q{5} + \frac{\q{3}}{\q{2}} \right)-1$& $\sqrt{\frac{1}{\q{3}}}$& $ \sqrt{\frac{\q{4}}{\q{2}^3}}  $&$\frac{1}{\sqrt{2}}$\\ \hline\hline
  Numerical Value & 0.745315& 0.790471 &0.800893 &0.8556  &0.865966 \\\hline
  Exact Guess      & $\sqrt{\frac{1}{\q{3}}\left( 1 + \frac{1}{\q{2}}  \right)}$&  $\sqrt{\frac{1}{\q{3}}\left( 1 + \frac{\q{2}}{\q{3}}  \right)}$&$\sqrt{\frac{\q{3}}{\q{2}\q{4}}\left( 1 + \frac{1}{\q{2}}  \right)}$& $\sqrt{\frac{ \q{3}}{\q{2}^2}}$& $\sqrt{ \frac{\q{2}}{\q{4}}\left(1 + \frac{1}{\q{4}} \right)  }$\\ \hline\hline
  Numerical Value & 0.896575& 0.975056 &1.020367 &1.035276 &1.07313 \\\hline
  Exact Guess      & $\frac{\q{4}}{\q{5}}$&  $\frac{1}{\q{5}} + \frac{\q{2}}{\q{3}}$&$\frac{\q{3}}{\q{2}\q{4}}\left(  1 + \frac{\q{3}}{\q{2}} \right)$& $\frac{1}{\q{3}}\left(\q{2} + \frac{\q{4}}{\q{5}} \right)$&$\sqrt{ \frac{1}{\q{2}}\left(1 + \frac{\q{4}}{\q{3}} \right)  }$\\ \hline\hline
  Numerical Value & 1.207107& 1.239146 &1.393847 &1.41421  &1.692705 \\\hline
  Exact Guess      & $\frac{\q{2}}{\q{3}}\left(1 + \frac{\q{2}}{\q{3}} \right)$&  $\frac{\q{3}}{\q{4}}\left(1 + \frac{1}{\q{2}} \right)$&$\frac{1}{\q{4}}\left(\q{2} + \q{3} \right)$& $\sqrt{2}$&$\frac{\q{2}}{\q{4}}\left(1 + \q{2} \right)$\\ \hline\hline
  Numerical Value & 1.93185\\\hline
  Exact Guess      & $\q{2}$\\ \hline\hline
\end{tabular}}\end{center}

Note that we could not (at this point) make guesses for the exact values of the following numerical values appearing in the solution
\[\{0.239691,0.301034,0.327424,0.420187,0.45152,0.465762,0.578665,0.633931,0.636242,0.72676\}.\]

Many of these exact values are for diagonal coefficients of our solution. i.e. coefficients of the form $U^{v_1,{}^{\gamma_1}v_2^{\gamma_2}}_{{}^{\gamma_1}v_2^{\gamma_2},v_3}$. These coefficients are real by (R2), and are gauge invariant. Hence the phase of these coefficients is $\pm 1$, and can easily be seen from the numerical solution. For the off-diagonal coefficients, we have a large amount of freedom in the phases due to the $U(1)^{25}$ gauge group. Using up this freedom allows us to set the phases of many of these coefficients to $1$.

We now have (a guess for) the exact value of many of the coefficients for our solution. At this point many of the equations of (Hecke) and (R3) are now linear. Solving these linear equations gives exact values for the remaining coefficients for which we could not guess an exact value for. Furthermore, these linear equations allow us to pin down all the remaining free phases in our solution.

As we now have an exact value for all of the coefficients in our solution, we thus have an element of $\Hom_{oGPA(\mathcal{E}_4^{12})}(++\to ++)$, which we suspect to be the embedding of $\att{UU}{.25}$ in $oGPA(\mathcal{E}_4^{12})$. We then obtain a potential embedding of $\att{Tri1}{.15}$ in $oGPA(\mathcal{E}_4^{12})$ by solving the linear equations (RI) and (BA), and normalising with (U).

The potential solution to the embedding of $\att{Tri1}{.15}\mapsto \Hom_{oGPA(\mathcal{E}_4^{12})}(-\to ++)$ is given in Definition~\ref{def:main}. Here we use a slight alteration of Boltzmann weight notation, with the value $W_{v_1,v_2,v_3}$ representing the coefficient of the basis element $( v_1 \xleftarrow{\gamma_3} v_3   ,v_1 \xrightarrow{\gamma_1} v_2 \xrightarrow{\gamma_2} v_3  )$ with edge labels surpessed unless needed.

To give the reader some idea of the structure of the solution for the embedding of $\att{UU}{.25}$, we include the single $5\times 5$ block and three $3\times 3$ blocks.

\[\resizebox{\hsize}{!}{$ U^{9}_{\quad 6}  = \begin{blockarray}{ccccc}
1 & 2&3&4&5 \\
\begin{block}{[ccccc]}
\frac{\q{2}}{\q{3}} & z\frac{1}{\q{3}} & z^{-1}\sqrt{\frac{\q{4}}{\q{2}\q{3}^2}}&\sqrt{\frac{\q{4}}{\q{2}\q{3}^2}}& -\mathbf{i}z \frac{1}{\q{2}}\\
z^{-1}\frac{1}{\q{3}}& \frac{1 + \q{3}}{\q{2}\q{3}\q{5}}&\sqrt{\frac{\q{4}}{\q{5}^3\q{3}^2}}\left(z^{-2} + z^{-1} \cdot \zeta_{24}^{11}\sqrt{\frac{\q{3}^2}{\q{2}\q{4}(\q{2}+\q{3})}}   \right)& \sqrt{\frac{\q{4}}{\q{2}\q{3}^2\q{5}}}\left(z^{-1} + \zeta_6 \sqrt{   \frac{\q{3}^2(\q{2}+\q{3})}{\q{2}\q{4}}}   \right)& \frac{1}{\q{2}^2}\left(  -\mathbf{i} + z^{-1}\zeta_{16}^9\sqrt{\frac{\q{3}\q{4}}{\q{2}}}   \right)\\
z\sqrt{\frac{\q{4}}{\q{2}\q{3}^2}}&\sqrt{\frac{\q{4}}{\q{5}^3\q{3}^2}}\left(z^{2} + z^{1} \cdot \zeta_{24}^{13}\sqrt{\frac{\q{3}^2}{\q{2}\q{4}(\q{2}+\q{3})}}   \right)& \frac{\q{4}}{\q{2}\q{3}}\left( \frac{1}{\q{2}} + \frac{\q{3}}{(\q{2}+\q{3})\q{4}}   \right)& \frac{\q{4}}{\q{3}\q{5}}\left(z+z\cdot \zeta_{24}^{17}\frac{\q{3}}{\q{4}}  \right)& \sqrt{\frac{\q{4}}{\q{2}\q{5}^2}}\left(  -\mathbf{i} z^2 + \zeta_{48}^{5}\sqrt{\frac{\q{3}}{\q{2}+\q{3}}}\right)\\
\sqrt{\frac{\q{4}}{\q{2}\q{3}^2}}& \sqrt{\frac{\q{4}}{\q{2}\q{3}^2\q{5}}}\left(z + \zeta_6^5 \sqrt{   \frac{\q{3}^2(\q{2}+\q{3})}{\q{2}\q{4}}}   \right)& \frac{\q{4}}{\q{3}\q{5}}\left(z^{-1}+z^{-1}\cdot \zeta_{24}^{7}\frac{\q{3}}{\q{4}}  \right)& \frac{\q{4}}{\q{2}^2\q{3}}\left( 1 + \frac{\q{3}(\q{2}+\q{3})}{\q{4}}\right) & \sqrt{\frac{\q{4}}{\q{2}\q{5}^2}}\left( -\mathbf{i}z + z^{-1} \cdot \zeta_{48}^{19}\sqrt{\frac{\q{3}(\q{2}+\q{3}}{\q{5}}}  \right)\\
\mathbf{i}z^{-1} \frac{1}{\q{2}}&\frac{1}{\q{2}^2}\left(  \mathbf{i} + z\cdot \zeta_{16}^7\sqrt{\frac{\q{3}\q{4}}{\q{2}}}   \right)&\sqrt{\frac{\q{4}}{\q{2}\q{5}^2}}\left(  \mathbf{i} z^{-2} + \zeta_{48}^{43}\sqrt{\frac{\q{3}}{\q{2}+\q{3}}}\right)&\sqrt{\frac{\q{4}}{\q{2}\q{5}^2}}\left( \mathbf{i}z^{-1} + z^{} \cdot \zeta_{48}^{29}\sqrt{\frac{\q{3}(\q{2}+\q{3}}{\q{5}}}  \right)& \frac{\q{3}}{\q{2}\q{5}}\left(    1 + \frac{\q{4}}{\q{2}}\right)\\
\end{block} \end{blockarray}$} \]
\[ U^{3}_{\quad 9} =   \begin{blockarray}{ccc}
6^\alpha & 6^\beta&10\\
\begin{block}{[ccc]}
\frac{1}{\q{2}} & \zeta_3^2 \sqrt{\frac{(1+\q{2})(\q{4}(\q{2}-1) - \q{3})}{\q{2}^2\q{4}}} &   \sqrt{\frac{\q{3}}{\q{2}\q{4}}\left( 1 + \frac{1}{\q{2}}  \right)} \\
\zeta_3 \sqrt{\frac{(1+\q{2})(\q{4}(\q{2}-1) - \q{3})}{\q{2}^2\q{4}}} & \frac{(1+\q{2})(\q{2}^2-\q{4})}{\q{2}\q{4}} & \zeta_3 \frac{1 + \q{2}}{\q{2}\q{4}}\sqrt{\q{3}(\q{5}-\q{4})}\\
 \sqrt{\frac{\q{3}}{\q{2}\q{4}}\left( 1 + \frac{1}{\q{2}}  \right)} & \zeta_3^2 \frac{1 + \q{2}}{\q{2}\q{4}}\sqrt{\q{3}(\q{5}-\q{4})} & \frac{\q{3}}{\q{4}} + \frac{\q{3}}{\q{2}\q{4}}\\
\end{block} \end{blockarray}    \]
\[ U^{4}_{\quad 9} =   \begin{blockarray}{ccc}
6^\alpha & 6^\beta&10\\
\begin{block}{[ccc]}
\frac{1}{\q{2}} &\zeta_8^3 \sqrt{\frac{\q{3}(\q{2}+\q{3})}{\q{2}\q{4}\q{5}}}&  z^{-1}\sqrt{\frac{\q{3}^2}{\q{2}^2\q{4}(1 +\q{2})}} \\
\zeta_8^5 \sqrt{\frac{\q{3}(\q{2}+\q{3})}{\q{2}\q{4}\q{5}}} &\frac{\q{3}}{\q{2}\q{4}}\left(1+\frac{\q{3}}{\q{2}}  \right) & z^{-1}\zeta_8^{5}\sqrt{ \frac{\q{3}^3(\q{2}+\q{3}}{\q{2}\q{4}^2\q{5}(1 + \q{2})} }\\
z\sqrt{\frac{\q{3}^2}{\q{2}^2\q{4}(1 +\q{2})}} & z\zeta_8^{3}\sqrt{ \frac{\q{3}^3(\q{2}+\q{3}}{\q{2}\q{4}^2\q{5}(1 + \q{2})} } & \frac{\q{2}\q{4}(\q{2}^2-1) - \q{3}(\q{2}+\q{3}}{\q{4}\q{5}}\\
\end{block} \end{blockarray}    \]
\[ U^{5}_{\quad 9} =   \begin{blockarray}{ccc}
6^\alpha & 6^\beta&10\\
\begin{block}{[ccc]}
\frac{1}{\q{2}}  &\zeta_{48}z^{-2} \sqrt{\frac{\q{4}}{\q{2}\q{5}}} & \mathbf{i}z^{-2} \frac{1}{\q{2}}\\ 
\zeta_{48}^{-1}z^{2}\sqrt{\frac{\q{4}}{\q{2}\q{5}}}  &\frac{\q{4}}{\q{5}} &  \zeta_{48}^{11}\sqrt{\frac{\q{4}}{\q{2}\q{5}}} \\ 
-\mathbf{i}z^{2} \frac{1}{\q{2}}  & \zeta_{48}^{37}\sqrt{\frac{\q{4}}{\q{2}\q{5}}}  &  \frac{1}{\q{2}}\\
\end{block} \end{blockarray}    \]
The remaining blocks can be recovered from the embedding of $\att{Tri1}{.15}$ using the formula
\[\att{UU}{.35} := \q{2}\att{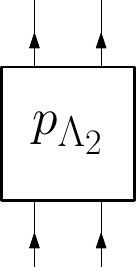}{.35} \quad=\quad  \att{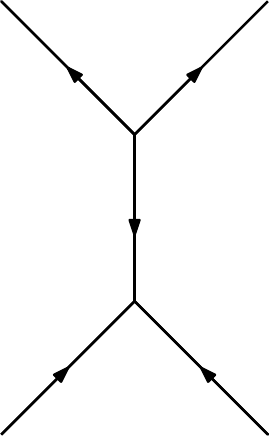}{.2}    .\]

\section{Verifying the solution}

From the previous section, we have a candidate for an exact solution to an embedding $\mathcal{P}_{\overline{\Rep(U_q(\mathfrak{sl}_3))}; \Lambda_1}\to oGPA(\mathcal{E}_4^{12})$. To show that this candidate solution is indeed an exact solution, we have to verify equations (R1), (R2), (Hecke), (RI), (BA), (U), and (R3). The first six are quickly verified by Mathematica. However, to verify (R3) we need to check that 1251 individual cubic equations hold. As the degree (as an algebraic number) of many of our coefficients is 16, these equations are extremely computationally expensive to verify. Optimistically we ran a computer for 72 hours attempting to verify the (R3) equations, however no result was obtained.

To get around this computational roadblock, we observe that the coefficients of the embedding of $\att{Tri1}{.15}$ are significantly nicer than the coefficients of the embedding of $\att{UU}{.25}$. As shown in \cite{Kup}, the category $\mathcal{P}_{\Rep(U_q(\mathfrak{sl}_3)); \Lambda_1}$ has an alternate presentation given in terms of the single generator $\att{Tri1}{.15}$. The relations of this presentation are as follows:
\begin{align*}
\textrm{()}&: \quad \att{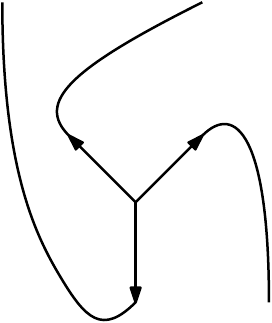}{.3} \quad=\quad \att{Tri1}{.3}\\
   \textrm{(i)}&: \quad \att{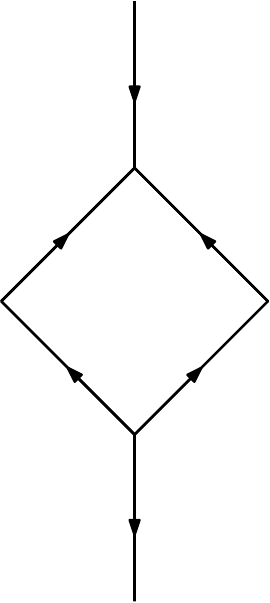}{.15} \quad=\quad \q{2}\att{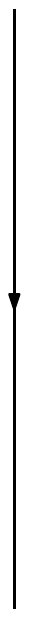}{.15}\\
    \textrm{(ii)}&: \quad \att{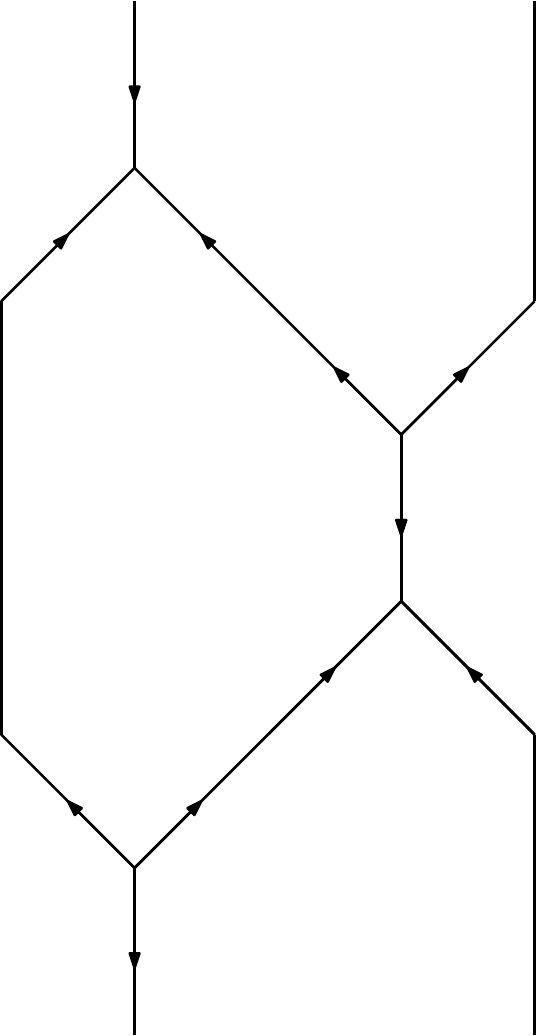}{.15} \quad=\quad \att{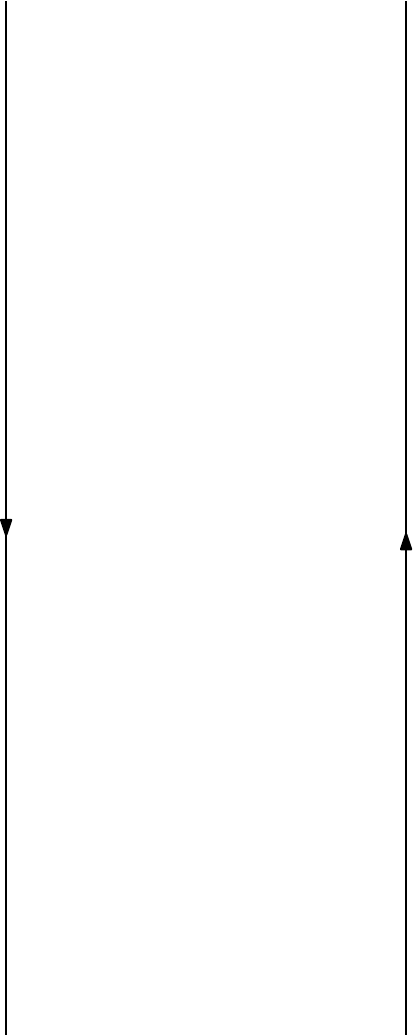}{.15} \quad+\quad\att{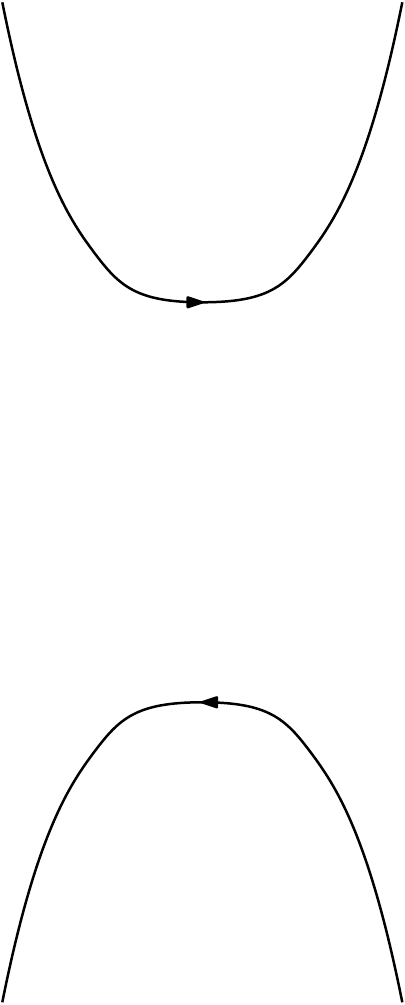}{.15} \\
\end{align*}
Hence if we can verify the above three relations, we will show that our potential solution indeed defines an embedding $\mathcal{P}_{\overline{\Rep(U_q(\mathfrak{sl}_3))}; \Lambda_1}\to oGPA(\mathcal{E}_4^{12})$. While relation (ii) is quartic, the simpler form of the algebraic numbers for the coefficients of the embedding of $\att{Tri1}{.2}$ means that these equations are much easier for the computer to verify. Helping our cause is the fact that there are only 171 individual equations to verify for relation (ii). This allows us to give a proof of Theorem~\ref{thm:main}.
\begin{proof}[Proof of Theorem~\ref{thm:main}]
We directly verify that the element of $oGPA(\mathcal{E}_4^{12})$ given in Definition~\ref{def:main} satisfies relations (), (i), and (ii) using a computer. A Mathematica file containing the solution and this verification can be found attached to the arXiv submission of this paper. This gives a $\dag$-embedding of $\mathcal{P}_{\Rep(U_q(\mathfrak{sl}_3)); \Lambda_1}\to oGPA(\mathcal{E}_4^{12})$. As $oGPA(\mathcal{E}_4^{12})$ is unitary, we have that the image of $\mathcal{P}_{\Rep(U_q(\mathfrak{sl}_3)); \Lambda_1}$ in $oGPA(\mathcal{E}_4^{12})$ is a unitary subcategory. In particular all negligible elements of $\mathcal{P}_{\Rep(U_q(\mathfrak{sl}_3)); \Lambda_1}$ are mapped to zero. Thus we get an embedding $\mathcal{P}_{\overline{\Rep(U_q(\mathfrak{sl}_3))}; \Lambda_1}\to oGPA(\mathcal{E}_4^{12})$ as desired.
\end{proof}
\bibliography{main}

\newcommand{\etalchar}[1]{$^{#1}$}
\begin{thebibliography}{GMP{\etalchar{+}}18}

\bibitem[BK01]{Bak}
Bojko Bakalov and Alexander Kirillov, Jr.
\newblock {\em Lectures on tensor categories and modular functors}, volume~21
  of {\em University Lecture Series}.
\newblock American Mathematical Society, Providence, RI, 2001.

\bibitem[CEM23]{dan}
Daniel Copeland and Cain Edie-Michell.
\newblock Cell systems for
  {$\overline{\operatorname{Rep}(U_q(\mathfrak{sl}_N))}$} module categories,
  2023.
\newblock \arxiv{2301.13172}.

\bibitem[Dav16]{UnPhys}
Alexei Davydov.
\newblock Unphysical diagonal modular invariants.
\newblock {\em J. Algebra}, 446:1--18, 2016.

\bibitem[DFZ90]{DiFran}
P.~Di~Francesco and J.-B. Zuber.
\newblock {${\rm SU}(N)$} lattice integrable models associated with graphs.
\newblock {\em Nuclear Phys. B}, 338(3):602--646, 1990.

\bibitem[EGNO15]{Book}
Pavel Etingof, Shlomo Gelaki, Dmitri Nikshych, and Victor Ostrik.
\newblock {\em Tensor categories}, volume 205 of {\em Mathematical Surveys and
  Monographs}.
\newblock American Mathematical Society, Providence, RI, 2015.

\bibitem[EM23]{ModPt1}
Cain Edie-Michell.
\newblock Type {II} quantum subgroups of {$\mathfrak{sl}_N$}. {I}: {S}ymmetries
  of local modules.
\newblock {\em Comm. Amer. Math. Soc.}, 3:112--165, 2023.
\newblock With appendix by Terry Gannon.

\bibitem[EP09a]{SU3}
David~E. Evans and Mathew Pugh.
\newblock Ocneanu cells and {B}oltzmann weights for the {$\rm SU(3)$} {${ADE}$}
  graphs.
\newblock {\em M\"{u}nster J. Math.}, 2:95--142, 2009.

\bibitem[EP09b]{SU3Sub}
David~E. Evans and Mathew Pugh.
\newblock {\rm {SU}}(3)-{G}oodman-de la {H}arpe-{J}ones subfactors and the
  realization of {\rm {su}}(3) modular invariants.
\newblock {\em Rev. Math. Phys.}, 21(7):877--928, 2009.

\bibitem[EP21]{EvansSO3}
David~E. Evans and Mathew Pugh.
\newblock Classification of module categories for {$SO (3)_{2m}$}.
\newblock {\em Adv. Math.}, 384:Paper No. 107713, 63, 2021.

\bibitem[Gan94]{SU3Mod}
Terry Gannon.
\newblock The classification of affine {${\rm SU}(3)$} modular invariant
  partition functions.
\newblock {\em Comm. Math. Phys.}, 161(2):233--263, 1994.

\bibitem[GMP{\etalchar{+}}18]{EH}
Pinhas Grossman, Scott Morrison, David Penneys, Emily Peters, and Noah Snyder.
\newblock The extended {H}aagerup fusion categories, 2018.
\newblock \arxiv{1810.06076}.

\bibitem[Jon00]{OGGPA}
Vaughan F.~R. Jones.
\newblock The planar algebra of a bipartite graph.
\newblock In {\em Knots in {H}ellas '98 ({D}elphi)}, volume~24 of {\em Ser.
  Knots Everything}, pages 94--117. World Sci. Publ., River Edge, NJ, 2000.

\bibitem[Jon22]{PA1}
V.~F.~R. Jones.
\newblock Planar algebras, {I}.
\newblock {\em New Zealand J. Math.}, 52:1--107, 2021 [2021--2022].

\bibitem[Kup96]{Kup}
Greg Kuperberg.
\newblock Spiders for rank {$2$} {L}ie algebras.
\newblock {\em Comm. Math. Phys.}, 180(1):109--151, 1996.

\bibitem[KW93]{SovietHans}
David Kazhdan and Hans Wenzl.
\newblock Reconstructing monoidal categories.
\newblock In {\em I. {M}. {G}elfand {S}eminar}, volume~16 of {\em Adv. Soviet
  Math.}, pages 111--136. Amer. Math. Soc., Providence, RI, 1993.

\bibitem[Mor10]{scottGPA}
Scott Morrison, 2010.
\newblock
  \href{https://tqft.net/papers/gpa.pdf}{https://tqft.net/papers/gpa.pdf}.

\bibitem[Ocn02]{Ocneanu}
Adrian Ocneanu.
\newblock The classification of subgroups of quantum {${\rm SU}(N)$}.
\newblock In {\em Quantum symmetries in theoretical physics and mathematics
  ({B}ariloche, 2000)}, volume 294 of {\em Contemp. Math.}, pages 133--159.
  Amer. Math. Soc., Providence, RI, 2002.

\bibitem[Pet10]{Haag}
Emily Peters.
\newblock A planar algebra construction of the {H}aagerup subfactor.
\newblock {\em Internat. J. Math.}, 21(8):987--1045, 2010.

\bibitem[Saw06]{Sawin}
Stephen~F. Sawin.
\newblock Quantum groups at roots of unity and modularity.
\newblock {\em J. Knot Theory Ramifications}, 15(10):1245--1277, 2006.

\bibitem[Sch20]{primer}
Andrew Schopieray.
\newblock Lie theory for fusion categories: a research primer.
\newblock In {\em Topological phases of matter and quantum computation}, volume
  747 of {\em Contemp. Math.}, pages 1--26. Amer. Math. Soc., [Providence], RI,
  [2020] \copyright 2020.

\bibitem[TV17]{Tur}
Vladimir Turaev and Alexis Virelizier.
\newblock {\em Monoidal categories and topological field theory}, volume 322 of
  {\em Progress in Mathematics}.
\newblock Birkh\"{a}user/Springer, Cham, 2017.

\end{thebibliography}
\bibliographystyle{alpha}
\end{document}